\documentclass[12pt,twoside]{amsart}
\usepackage{latexsym,amsmath,amsopn,amssymb,amsthm,amsfonts}
\usepackage[T2A]{fontenc}
\usepackage[cp1251]{inputenc}
\usepackage[russian, english]{babel}
\usepackage{graphicx}

\newtheorem{theorem}{Theorem}

\newtheorem{proposition}{Proposition}

\newtheorem{remark}{Remark}

\newtheorem{example}{Example}

\newtheorem{lemma}{Lemma}

\begin{document}

	\title[sub-Finsler Engel group]{Extremals of a left-invariant sub-Finsler metric on the Engel group}
	\author{V.~N.~Berestovskii, I.~A.~Zubareva}
	\thanks{The work is supported by Mathematical Center in Akademgorodok under agreement No. 075-15-2019-1613 with the
	Ministry of Science and Higher Education of the Russian Federation}
	\address{Sobolev Institute of Mathematics, \newline
		Acad. Koptyug avenue, 4, Novosibirsk, 630090, Russia}
	\address{Novosibirsk State University,\newline
			Pirogov str, 1, Novosibirsk, 630090, Russia}
			\email{vberestov@inbox.ru}
	\address{Sobolev Institute of Mathematics,\newline
		Pevtsova str., 13, Omsk, 644099, Russia}
	\email{i\_gribanova@mail.ru}
	\begin{abstract}
Using the Pontryagin Maximum Principle for the time-optimal problem in coordinates of the first kind, we find extremals of abitrary left--invariant sub--Finsler metric on the Engel group defined by a distribution of rank two.
		
		\vspace{2mm}
	
	\end{abstract}
	\maketitle
	
	\section*{Introduction}
	
In \cite{Ber1}, it is indicated that the shortest arcs of any left-invariant (sub-)Finsler metric $d$ on a Lie group $G$ are
solutions of a left-invariant time-optimal problem with the closed unit ball $U$ of some {\it arbitrary} norm $F$ on a
subspace $\mathfrak{p}$ of the Lie algebra $(\mathfrak{g},[\cdot,\cdot])$ of the Lie group $G$ as a control region. In addition,
the subspace $\mathfrak{p}$ generates $\mathfrak{g}$. The Pontryagin Maximum Principle gives the necessary conditions for
optimal trajectories of the problem \cite{PBGM}; the curves, satisfying these conditions, are called {\it extremals}. Apparently, for the first time the shortest arcs of any left-invariant sub-Finsler metric on Lie group have been found in paper \cite{Ber2} in the case
of arbitrary sub-Finsler metric $d$ on the Heisenberg group $H$. The quotient group of $H$ by its center $Z$ is isomorphic to
the additive group $(\mathbb{R}^2,+).$ Moreover, the differential $dp$ of the canonical projection $p: H\rightarrow H/Z=\mathbb{R}^2$
is a linear isomorphism of the subspace $\mathfrak{p}$ onto $T_0\mathbb{R}^2=\mathbb{R}^2.$ The identification of the spaces $\mathfrak{p}$ and $\mathbb{R}^2$ by means of $dp$ turns $\mathbb{R}^2$ into a normed vector space $(\mathbb{R}^2,F),$ the so-called
Minkowski plane. In \cite{Ber2}, with the help of the mentioned maximum principle, it is proved that the projection with respect to
$p$ of any maximal by inclusion shortest curve in $(H,d)$ can be part of 1) a metric straight line or 2) a (closed) {\it isoperimetrix} \cite{Lei} of the Minkowski plane $(\mathbb{R}^2,F).$ 
	
Earlier in \cite{Bus}, H.~Busemann obtained the solution to the isoperimetric problem for the Minkowski plane. With a reference to \cite{Ber2}, G.A.~Noskov founds in \cite{Nos} 
the same shortest curves in $(H,d)$ on the base of \cite{Bus} and some nontrivial argument. On the other hand,  
the statement in \cite{AS5} and \cite{Loc} that Busemann found in \cite{Bus} the shortest curves of the space $(H,d)$ is erroneous. 
This is not only because at that time there was no equivalent to sub-Finsler geometry, but also because the shortest paths of the type 1) mentioned above are not connected with the isoperimetrix.
The authors of \cite{CM} (see also \cite{CMW}) supposed that they were the first who studied sub-Finsler manifolds.
But with the other name (homogeneous) "nonholonomic Finsler manifolds", they appeared yet in three works of the first author published in 1988 and 1989, including \cite{Ber1}, in connection with a characterization of {\it general homogeneous manifolds with inner metric}. Besides this, following the tradition of specialists in Finsler geometry, the authors of papers \cite{CM} and \cite{CMW} superpose additional strong conditions on the norm $F$ and apply the corresponding cumbersome apparatus.   
	
In this paper we find extremals of arbitrary left-invariant sub-Finsler metric on the Engel group, defined by a subspace  $\mathfrak{p}$ of rank two.  In papers \cite{AS1}--\cite{AS4} Ardentov and Sachkov investigated in detail left-invariant sub-Riemannian metric on the Engel group in coordinates different from ones of the fist kind which we apply.

Only classical methods and results from the monograph \cite{PBGM} are applied here. 
In paper \cite{BerZub} are proposed some new search methods of normal 
extremals of left-invariant (sub-)Finsler and (sub-)Riemannian metrics.

The authors thank  L.~V.~Lokutsievskiy for useful discussions.
	
\section{The Campbell-Hausdorff formula for the Engel group}
	
Let $X$, $Y$, $Z$, $V$ be a basis of the four-dimensional Engel algebra $\mathfrak{g}$ such that
\begin{equation}
\label{a1}
[X,Y]=Z,\quad [X,Z]=V,\quad [X,V]=[Y,V]=[Z,V]=[Y,Z]=0.
\end{equation}
Thus $\mathfrak{g}$ is a three-step nilpotent Lie algebra with two generators $X$, $Y$. Therefore, as it is known, there
exists a unique up to isomorphism connected simply connected nilpotent Lie group $G$ with the Lie algebra $\mathfrak{g},$ 
the Engel group, and the exponential mapping $\exp: \mathfrak{g}\rightarrow G$ is a diffeomorphism. This diffeomomorphism
and the Cartesian coordinates $x,\,y,\,z,\,v$ in  $\mathfrak{g}$ with the basis $X$, $Y$, $Z$, $V$ defines the
coordinates of the first kind on $G$ and thus a diffeomorphism $G\cong\mathbb{R}^4$.
	
	\begin{proposition}
		\label{product}
In the coordinates of the first kind, the multiplication on the Engel group $G\cong\mathbb{R}^4_{x,y,z,v}$ is given by 
the following rule
\begin{equation}
\label{a2}
\left(\begin{array}{c}
x_1 \\
y_1 \\
z_1 \\
v_1
\end{array}\right)\times\left(\begin{array}{c}
x_2 \\
y_2 \\
z_2 \\
v_2
\end{array}\right)=\left(\begin{array}{c}
x_1+x_2 \\
y_1+y_2 \\
z_1+z_2+\frac{1}{2}(x_1y_2-x_2y_1) \\
v_1+v_2+\frac{1}{2}(x_1z_2-x_2z_1)+\frac{1}{12}(x_1^2y_2-x_1x_2y_2-x_1x_2y_1+x_2^2y_1)
\end{array}\right).
\end{equation}
\end{proposition}
	
\begin{proof}
Set $A_i=x_iX+y_iY+z_iZ+v_iV$, $i=1,2.$ Using (\ref{a1}), we consequently obtain
$$[A_1,A_2]=(x_1y_2-x_2y_1)[X,Y]+(x_1z_2-x_2z_1)[X,Z]=$$
$$(x_1y_2-x_2y_1)Z+(x_1z_2-x_2z_1)V;$$
$$[A_1,[A_1,A_2]]=x_1(x_1y_2-x_2y_1)[X,Z]=x_1(x_1y_2-x_2y_1)V;$$
$$[A_2,[A_2,A_1]]=[[A_1,A_2],A_2]=x_2(x_1y_2-x_2y_1)[Z,X]=x_2(x_2y_1-x_1y_2)V.$$
		
Since the Lie algebra $\mathfrak{g}$ is of three-step, then it is valid the following Campbell-Hausdorff formula 
(see \cite{Post}):		$$\ln\left(\exp(A_1)\exp(A_2)\right)=A_1+A_2+\frac{1}{2}[A_1,A_2]+\frac{1}{12}[A_1,[A_1,A_2]]+\frac{1}{12}[A_2,[A_2,A_1]].$$
Therefore
$$\ln\left(\exp(A_1)\exp(A_2)\right)=(x_1+x_2)X+ (y_1+y_2)Y+\left(z_1+z_2+\frac{1}{2}(x_1y_2-x_2y_1)\right)Z+$$
$$\left(v_1+v_2+\frac{1}{2}(x_1z_2-x_2z_1)+\frac{1}{12}(x_1y_2-x_2y_1)(x_1-x_2)\right)V.$$
The last equality gives (\ref{a2}).
\end{proof}
	
It follows from the method we introduced the coordinates of the first kind and formulas (\ref{a2}) that the realization of the chosen basis of the Lie algebra $\mathfrak{g}$ as left-invariant vector fields on the Lie group $G$ in these coordinates has the form
\begin{equation}
\label{XYZV}
X= \frac{\partial}{\partial x} -\frac{y}{2}\frac{\partial}{\partial z} - \frac{z}{2}\frac{\partial}{\partial v} -\frac{xy}{12}\frac{\partial}{\partial v},\,\,
Y= \frac{\partial}{\partial y} + \frac{x}{2}\frac{\partial}{\partial z} + \frac{x^2}{12}\frac{\partial}{\partial v},\,\,
Z= \frac{\partial}{\partial z} + \frac{x}{2}\frac{\partial}{\partial v},\,\,
V = \frac{\partial}{\partial v}.
\end{equation}	
It is easy to verify that these vector fields satisfy relations (\ref{a1}).	
		
\section{Left-invariant sub-Finsler metric and the optimal control on the Engel group}
	
In \cite{Ber1}, it is said that  the shortest arcs of a left-invariant sub-Finsler metric $d$ on arbitrary connected Lie group	
$G$ defined by a left-invariant bracket generating distribution $D$ and a norm  $F$ on $D(e)$ coincide with the time-optimal solutions of
the following control system 	
\begin{equation}
\label{a3}
\dot{g}(t)=dl_{g(t)}(u(t)),\quad u(t)\in U,
\end{equation}
with measurable controls $u=u(t)$. Here $l_g(h)=gh$, the control region is the unit ball 
$$U=\{u\in D(e)\,|F(u)\leq 1\}.$$ 

Therein by virtue of the Pontryagin Maximum Principle for (local) time optimality of a control $u(t)$ and corresponding trajectory  $g(t),$ $t\in\mathbb{R}$, it is necessary the existence of a non-vanishing absolutely continuous vector-function
$\psi(t)\in T^{\ast}_{g(t)}G$ such that for almost all $t\in\mathbb{R}$ the function 
$\mathcal{H}(g(t);\psi(t);u)=\psi(t)(dl_{g(t)}(u))$ of the variable $u\in U$ attains the maximum at the point $u(t)$:
\begin{equation}
\label{m0}
M(t)=\psi(t)(dl_{g(t)}(u(t)))=\max\limits_{u\in U}\psi(t)(dl_{g(t)}(u)).
\end{equation}
In addition, the function $M(t)$, $t\in\mathbb{R}$, is constant and non-negative, $M(t)\equiv M\geq 0$.

In case when $M=0$ (respectively, $M>0$) the corresponding {\it extremal}, i.e. the curve, satisfying the Pontryagin Maximum Principle, is called {\it abnormal} (respectively, {\it normal}).

It follows from (\ref{a1}) that the left-invariant distribution $D$ on $G$ with the basis $X, Y$ for $D(e)$ is bracket generating. Let $F$ be an arbitrary norm on $D(e)$. Then the pair  $(D(e),F)$ defines a left-invariant sub-Finsler metric  
$d$ on $G$; therein $u_1X(e)+u_2Y(e)$ is identified with $u=(u_1,u_2),$ where $u_i\in\mathbb{R},$ $i=1,2.$
	
With regard to (\ref{a2}) the control system (\ref{a3}) is written as
\begin{equation}
\label{a4}
\dot{x}=u_1,\quad\dot{y}=u_2,\quad\dot{z}=\frac{1}{2}(xu_2-yu_1),
\quad\dot{v}=-\frac{1}{2}\left(z+\frac{1}{6}xy\right)u_1+\frac{1}{12}x^2u_2,
\end{equation}
where $(u_1,u_2)\in U$. 
	
{\it In consequence of left-invariance of the metric $d$ we can assume that the trajectories initiate at the unit $e\in G$, i.e. $x(0)=y(0)=z(0)=v(0)=0$}.
	
The control $u=u(t)=(u_1(t),u_2(t))\in U,$ $t\in \mathbb{R},$ defined by the Pontryagin Maximum Principle is bounded and measurable \cite{PBGM}, therefore integrable. Then the functions $x(t),$ $y(t),$ $t\in\mathbb{R},$ defined by the first two equations in (\ref{a4}) are Lipschitz, the product of any finite number	of these functions is Lipschitz, and its derivative is bounded and measurable on each compact segment of $\mathbb{R}$. Moreover, this derivative can be calculated by the usual differentiation rule of a product from differential calculus of 
functions of one variable. Therefore, the last two equations of system (\ref{a4}) can be integrated in parts, using the first two equations in (\ref{a4}) (see ss. 2.9.21, 2.9.24 in \cite{Fed}). Taking into account $x(0)=y(0)=z(0)=v(0)=0$
we sequentially get
\begin{equation}
\label{aa4}
z(t)=-\frac{1}{2}x(t)y(t)+\int\limits_0^t x(\tau)u_2(\tau)d\tau,
\end{equation}
$$v(t)=-\frac{1}{2}x(t)\left(z(t)+\frac{1}{3}x(t)y(t)\right)+\frac{1}{2}\int\limits_0^tx^2(\tau)u_2(\tau)d\tau=$$
\begin{equation}
\label{aa5}
\frac{1}{12}x^2(t)y(t)-\frac{1}{2}x(t)\int\limits_0^t x(\tau)u_2(\tau)d\tau+\frac{1}{2}\int\limits_0^tx^2(\tau)u_2(\tau)d\tau.
\end{equation}
	
According to the Pontryagin Maximum Principle, the system (\ref{a4}) corresponds to a function $\mathcal{H}(x,y,z,v;\psi_1,\psi_2,\psi_3,\psi_4;u_1,u_2)$ defined  by formula
$$\mathcal{H}=\psi_1u_1+\psi_2u_2+\frac{1}{2}\psi_3(xu_2-yu_1)-
\frac{1}{2}\psi_4\left(z+\frac{1}{6}xy\right)u_1+\frac{1}{12}\psi_4x^2u_2=h_1u_1+h_2u_2,$$
where
\begin{equation}
\label{a5}
h_1=\psi_1-\frac{1}{2}\psi_3y-\frac{1}{12}\psi_4xy-\frac{1}{2}\psi_4z,\quad h_2=\psi_2+\frac{1}{2}\psi_3 x+\frac{1}{12}\psi_4 x^2.
\end{equation}
	
The absolutely continuous vector-function $\psi=\psi(t)$ satisfies the conjugate to (\ref{a4}) system of ordinary differential
equations
\begin{equation}
\label{a6}
\left\{\begin{array}{l}
\dot{\psi_1}=\frac{1}{12}\psi_4yu_1-\left(\frac{1}{2}\psi_3+\frac{1}{6}\psi_4x\right)u_2,\\
\dot{\psi_2}=\left(\frac{1}{2}\psi_3+\frac{1}{12}\psi_4x\right)u_1,\\
\dot{\psi_3}=\frac{1}{2}\psi_4u_1,\\
\quad\dot{\psi_4}=0.
\end{array}\right.
\end{equation}
Assign an arbitrary set of initial data $\psi_i(0)=\varphi_i$, $i=1,2,3,4$, of the system (\ref{a6}). It follows from (\ref{a6}), 
the first equation in (\ref{a4}), and the initial condition $x(0)=0$ that  
\begin{equation}
\label{psi3}
\psi_4\equiv\varphi_4,\quad\psi_3=\varphi_3+\frac{1}{2}\varphi_4x,\quad\psi_2=\varphi_2+\frac{1}{2}\varphi_3x+\frac{1}{6}\varphi_4x^2.
\end{equation}
Notice that $\left(\frac{1}{2}xy+z\right)^{\cdot}=xu_2$, $\left(\frac{1}{2}xy-z\right)^{\cdot}=yu_1$ on the ground of (\ref{a4}). 
With regard to (\ref{psi3}) and (\ref{a4}) the first equation in (\ref{a6}) takes a form
$$\dot{\psi}_1=\frac{1}{12}\varphi_4\left(\frac{1}{2}xy-z\right)^{\cdot}-\frac{1}{2}\varphi_3\dot{y}-\frac{5}{12}\varphi_4\left(\frac{1}{2}xy+z\right)^{\cdot}.$$
Therefore, taking into account of the initial data of systems (\ref{a4}) and (\ref{a6}), we get 
\begin{equation}
\label{psi1}
\psi_1=\varphi_1-\frac{1}{2}\varphi_3y-\frac{1}{6}\varphi_4\left(xy+3z\right).
\end{equation} 
Inserting the last equality and (\ref{psi3}) into (\ref{a5}), we find
\begin{equation}
\label{h12}
h_1=\varphi_1-\left(\varphi_3+\frac{1}{2}\varphi_4x\right)y-\varphi_4z,\quad h_2=\varphi_2+\left(\varphi_3 +\frac{1}{2}\varphi_4 x\right)x.
\end{equation}

We notice, that $\psi_k,$ $k=1,2,3,4,$ are covector components of $\psi=\psi(t)$ relative to the coordinate system $(x,y,z,v),$ i.e.
\begin{equation}
\label{psic}
\psi_1=\psi\left(\frac{\partial}{\partial x}\right),\quad \psi_2=\psi\left(\frac{\partial}{\partial y}\right),\quad 
\psi_3=\psi\left(\frac{\partial}{\partial z}\right),\quad \psi_4= \psi\left(\frac{\partial}{\partial v}\right).
\end{equation}
Let $h_1=\psi(X),$ $h_2=\psi(Y),$ $h_3=\psi(Z),$ $h_4=\psi(V).$ Using (\ref{XYZV}), it is easy to verify that the formulas (\ref{psi3}), (\ref{psi1}), (\ref{psic}) give the same $h_1,$ $h_2,$ as in (\ref{h12}), and
\begin{equation}
\label{h34}
h_3= \psi_3 + \psi_4\frac{x}{2}=\varphi_3+\frac{1}{2}\varphi_4x + \varphi_4\frac{x}{2}=\varphi_3 +\varphi_4x,\quad h_4=\psi_4=\varphi_4.
\end{equation}
From (\ref{h12}) and (\ref{h34}) we obtain two more integrals of the Hamiltonian system (\ref{a4}), (\ref{a6}):
\begin{equation}
\label{c12}
h_4\equiv \varphi_4,\quad \mathcal{E}=\frac{h_3^2}{2}-h_2h_4\equiv \frac{\varphi_3^2}{2}- \varphi_2\varphi_4,
\end{equation}
called in \cite{AS5} {\it the Casimir functions}.

Now, using (\ref{a4}), (\ref{h12}) and (\ref{h34}),  we compute
\begin{equation}
\label{a7}
\dot{h}_1=-h_3u_2,\quad \dot{h}_2=h_3u_1.
\end{equation}
	
By virtue of the Pontryagin Maximum Principle for local time optimality of a control $u(t)$ and corresponding trajectory  $(x(t),y(t),z(t),v(t)),$ it is necessary the existence of a non-vanishing absolutely continuous vector-function $\psi(t)$ such that for almost all $t\in\mathbb{R}$ the ODE system (\ref{a6}) is satisfied and the function 
$$\mathcal{H}(x(t),y(t),z(t),v(t);\psi_1(t),\psi_2(t),\psi_3(t),\psi_4(t);u_1,u_2)$$
of the variable $u\in U$ attains the maximum at the point $u(t)$:
\begin{equation}
\label{m}
M(t)=h_1(t)u_1(t)+h_2(t)u_2(t)=\max\limits_{u\in U}(h_1(t)u_1+h_2(t)u_2).
\end{equation}
	
Relations (\ref{a4}), (\ref{h12}) and (\ref{m}) imply that under multiplication of functions $\psi_i(t)$, $i=1,2,3,4$, by a positive constant $k$ the trajectory $(x(t),y(t),z(t),v(t))$ does not change, while $M$ is multipled by $k$. Therefore {\it in case when $M>0$ we shall assume that} $M=1$. 
{\it Further in this section we consider this case}.
	
It follows from (\ref{m}) that $(h_1(t),h_2(t))$ from (\ref{h12}) and $(\varphi_1,\varphi_2)=(h_1(0),h_2(0))$ lie on the boundary
$\partial U^{\ast}$ of the polar figure $U^{\ast}=\{h\,|F_{U}(h)\leq 1\}$ to $U,$ where $F_{U}$ is a norm on $H=\{h\},$ equal
to the support Minkowski function of the body $U$:	
$$F_U(h)=\max\limits_{u\in U}h\cdot u.$$
In addition, $(H,F_U)$ is the conjugate normed vector space to $(D(e),F)$ and $(U^{\ast})^{\ast}=U$ in consequence of 
reflexivity of finite-dimensional normed vector spaces. Moreover, using (\ref{a7}) and (\ref{m}), we get
\begin{equation}
\label{area}
h_1(t)\dot{h}_2(t)-\dot{h}_1(t)h_2(t)=(\varphi_3+\varphi_4x(t))(h_1(t)u_1(t)+h_2(t)u_2(t))=\varphi_3+\varphi_4x(t).
\end{equation} 
	
Let $r=r(\theta)$, $\theta\in\mathbb{R}$, be a polar equation of the curve $F_U(x,y)=1$. At every point $\theta\in \mathbb{R}$
there exist one-sided derivatives of $r=r(\theta)$ (and with possible exclusion of no more than countable number of values
$\theta$ there exists the usual derivative $r'(\theta)$). For simplicity {\it we shall denote every value between these derivatives by} $r'(\theta)$. Then for $\theta=\theta(t),$
\begin{equation}
\label{h}
h_1(t)=h_1(\theta)=r(\theta)\cos\theta,\quad h_2(t)=h_2(\theta)=r(\theta)\sin\theta,
\end{equation}   
\begin{equation}
\label{hder}  
h'_1(\theta)=-(r(\theta)\sin\theta-r^{\prime}(\theta)\cos\theta),\quad h'_2(\theta)=(r^{\prime}(\theta)\sin\theta+r(\theta)\cos\theta).
\end{equation}
	
Independently on the existence of usual derivative (\ref{hder}), (\ref{area}) implies the existence of usual derivative 
for the doubled oriented area  
$$\sigma(t)=2S(\theta(t))=\int_0^{\theta(t)}r^2(\theta)d\theta$$ 
of the sector, counted from $0.$ In addition, by (\ref{h34}) and (\ref{area}) 
\begin{equation}
\label{derst}
\dot{\sigma}(t)=\varphi_3+\varphi_4x(t)=r^2(\theta(t))\dot{\theta}(t),\quad \dot{\theta}(t)=\frac{\dot{\sigma}(t)}{r^2(\theta(t))}.
\end{equation}
If we square the second equality in (\ref{derst}), we get by  (\ref{h12})
$$r^4(\theta)\dot{\theta}^2=\varphi_3^2+2\varphi_4\left(\varphi_3+\frac{1}{2}\varphi_4x\right)x=\varphi_3^2+2\varphi_4(h_2-\varphi_2),$$
\begin{equation}
\label{dt}
\dot{\theta}^2=\frac{\varphi_3^2+2\varphi_4(r(\theta)\sin\theta-\varphi_2)}{r^4(\theta)}.
\end{equation}
On the ground of (\ref{a4}), (\ref{a7}), and (\ref{derst}),
\begin{equation}
\label{pen}
\ddot{\sigma}(t)=\varphi_4u_1(t),
\end{equation}
\begin{equation}
\label{energy}
\mathcal{E}=\mathcal{E}(t)=\frac{1}{2}(\dot{\sigma}(t))^2-\varphi_4h_2(t)={\rm const}.
\end{equation}
	
\begin{remark}
(\ref{energy}) is equivalent to (\ref{dt}).	
\end{remark} 
	
\begin{remark}
In the notation from \cite{Loc}, the equation (\ref{pen}) is written as
\begin{equation}
\label{pen2}
\ddot{\theta^{\circ}}=\varphi_4\cos_{\Omega}\theta,
\end{equation}
it's analogue to the equation (5), when $\varphi_4\neq 0$: $\ddot{\theta^{\circ}}=\sin_{\Omega}\theta$ from \cite{Loc}.
In that paper $\Omega=U,$  $\Omega^{\circ}=U^{\ast};$ $\theta^{\circ}$ is our $\sigma(\theta)$ for $\Omega^{\circ},$
in \cite{Loc}, $\theta$  plays a role of $\sigma$ for $\Omega,$ 
$$\cos_{\Omega}\theta=u_1(\theta),\quad \sin_{\Omega}\theta=u_2(\theta),\quad \cos_{\Omega^{\circ}}\theta^{\circ}=h_1(\theta^{\circ}),\quad \sin_{\Omega^{\circ}}\theta^{\circ}=h_2(\theta^{\circ}),$$
$$\cos_{\Omega^{\circ}}\theta^{\circ}\cos_{\Omega}\theta + \sin_{\Omega^{\circ}}\theta^{\circ}\sin_{\Omega}\theta=1.$$  
Figure 8 in \cite{Loc} shows a schematic representation to a phase portrait of a "generalized mathematical pendulum" (5). 
On the basis of the portrait, there is also given some general verbal, but rather detailed, information on the solutions to equation (5) and its application (including) to the Heisenberg, the Cartan, and the Engel groups with the given norm $F$ on $D(e)$ for left-invariant two-dimensional totally nonholonomic distribution $D$ on these Lie groups. 
In other words, an analogue of our function $\theta=\theta(t)$ for these Lie groups is described in sufficient detail in \cite{Loc}. At the same time, the corresponding extremals on these groups are not searched in \cite{Loc}.  	
\end{remark} 	
	
We claim that {\it in general case} for $\theta=\theta(t),$	
\begin{equation}
\label{u}
\dot{x}(t)=u_1(\theta)=\frac{r^{\prime}(\theta)\sin\theta+r(\theta)\cos\theta}{r^2(\theta)},\quad
\dot{y}(t)=u_2(\theta)=\frac{r(\theta)\sin\theta-r^{\prime}(\theta)\cos\theta}{r^2(\theta)},
\end{equation}
where $u_1(\theta)=h'_2(\theta)/r^2(\theta),$ $u_2(\theta)=-h'_1(\theta)/r^2(\theta)$ due to (\ref{hder}).
Indeed, the following two equalities must hold: 
$$h_1(\theta)u_1(t)+h_2(\theta)u_2(\theta)=1,\quad h'_1(\theta)u_1(t)+h'_2(\theta)u_2(\theta)=0.$$
It is easy to see that the first of these equalities follows from (\ref{h}) and (\ref{u}), while the second is a corollary from (\ref{hder}) 
and (\ref{u}).
	
It follows from (\ref{a4}) that
\begin{equation}
\label{y}
\left(3v+\frac{1}{2}xz\right)^{\cdot}=-\frac{3}{2}\dot{x}z+\frac{1}{2}x\dot{z}+\frac{1}{2}\dot{x}z+\frac{1}{2}x\dot{z}=x\dot{z}-\dot{x}z,
\end{equation}
so on the base of (\ref{a4}), (\ref{h12}), and (\ref{m}) we get, omitting for brevity the variable $t$,
$$h_1u_1+h_2u_2=\varphi_1\dot{x}+\varphi_2\dot{y}+\left(\varphi_3+\frac{1}{2}\varphi_4x\right)(x\dot{y}-\dot{x}y)-
\varphi_4\dot{x}z=\left(\varphi_1x+\varphi_2y+2\varphi_3z\right)^{\cdot}+$$
$$\varphi_4(x\dot{z}-z\dot{x})=
\left(\varphi_1x+\varphi_2y+2\varphi_3z+3\varphi_4v+\frac{1}{2}\varphi_4xz\right)^{\cdot}=1.$$
Taking into account of the initial data of system (\ref{a4}), we obtain
\begin{equation}
\label{eq}
\varphi_1x+\varphi_2y+\left(2\varphi_3+\frac{1}{2}\varphi_4x\right)z+3\varphi_4v=t.
\end{equation}
	
\section{Search for sub-Finsler extremals}
\label{geod}
	
{\bf 1.} Let us consider an abnormal case. It is valid the following proposition.
	
\begin{proposition}
\label{ageod1}
An abnormal extremal on the Engel group starting at the unit is a one-parameter subgroup
\begin{equation}
\label{anorm}
x(t)\equiv 0,\quad y(t)=\pm\frac{t}{F(0,1)},\quad z(t)\equiv 0,\quad v(t)\equiv 0
\end{equation}
and is not strongly abnormal.
\end{proposition}
	
\begin{proof}
Assume that $M=0$. Then we obtain from the maximum condition that $h_1(t)=h_2(t)\equiv 0$ and $\varphi_1=\varphi_2=0$. Since $u_1(t)$ and $u_2(t)$ could not simultaneously vanish at any $t\in\mathbb{R}$, then $\varphi_3+\varphi_4x(t)\equiv 0$ on the base of (\ref{a7}). This implies that $\varphi_3=0$ and $x(t)\equiv 0$ because $x(0)=0$. Hence in consequence of (\ref{psi3}) we get 
$\psi_2(t)=\psi_3(t)\equiv 0,$ $\psi_4(t)\equiv \varphi_4,$ and $\psi_1(t)=\frac{1}{2}\varphi_4 z(t)$ on the ground of 
(\ref{psi1}) and the first equality in (\ref{h12}). Therefore, $\varphi_4\neq 0$ because $\psi(t)$ does not vanish.
		
Since $x(t)\equiv 0$, then $u_1(t)\equiv 0$ according to the first equality (\ref{a4}). Hence we obtain successively from
the third and the fourth equations in (\ref{a4}) as well as of the initial data $z(0)=v(0)=0$ that $z(t)=v(t)\equiv 0$. 
		
Further, since $u_1(t)\equiv 0$ and $F(u_1(t),u_2(t))\equiv 1$, then $u_2(t)\equiv\pm\frac{1}{F(0,1)}$. This, 
the second equation in (\ref{a4}), and the initial condition $y(0)=0$ imply that $y(t)=\pm\frac{t}{F(0,1)}$, and we get
(\ref{anorm}). 
		
In consequence of (\ref{a2}), this extremal is one of two one-parameter subgroups
$$g_1(t)=\exp\left(\frac{tY}{F(0,1)}\right),\quad g_2(t)=g_1(-t)=g_1(t)^{-1},\quad t\in\mathbb{R},$$ 
satisfies (\ref{m}) with $M(t)\equiv 1$ for constant covector function
$$\psi(t)=(0,\pm\varphi_2,0,0)=(0,\pm F(0,1),0,0)=(0,h_2(t),0,0),$$
subject to differential equations (\ref{a6}) and (\ref{a7}); therefore it is normal relative to this covector function, 
is not strongly abnormal, and is a geodesic, moreover, is a metric straight line (see Proposition \ref{norm1} below). 
\end{proof}
	
{\bf 2.} Set $M=1$. 
	
\begin{theorem}
\label{mt}	
For every extremal on the Engel group starting at the unit,  
\begin{equation}
\label{xt}
x(t)=\int_0^t\frac{[r^{\prime}(\theta(\tau))\sin\theta(\tau)+r(\theta(\tau))\cos\theta(\tau)]d\tau}{r^2(\theta(\tau))},
\end{equation} 
\begin{equation}
\label{yot}
y(t)=\int_0^t\frac{[r(\theta(\tau))\sin\theta(\tau)-r^{\prime}(\theta(\tau))\cos\theta(\tau)]d\tau}{r^2(\theta(\tau))}
\end{equation}
with arbitrary measureable integrands of indicated view and continuously differentiable function 
$\theta=\theta(t),$ satisfying (\ref{derst}), (\ref{dt}).   	
\end{theorem}
	
\begin{proof}
By Proposition \ref{ageod1}, every extremal is normal for corresponding control. 
On the ground of above assertions, any control has a view (\ref{u}) which implies (\ref{xt}), (\ref{yot}). 	
\end{proof}
	
\begin{remark}
On parts of the extremals with $\dot{\theta}(t)\neq 0$ for calculation of functions $x(t),$ $y(t)$ by formulas 
(\ref{xt}), (\ref{yot}) have matter only those values  $\theta=\theta(t)$ where the usual derivative $r'(\theta)$ exists.
\end{remark}	
	
{\bf 2.1.} Let us assume that $\varphi_3=\varphi_4=0$. The following proposition is true.
	
\begin{proposition}
\label{norm1}
For any extremal on the Engel group with above conditions and origin at the unit, $\theta(t)\equiv\theta_0,$ 
$t\in\mathbb{R},$ for some $\theta_0.$ In addition, every such extremal is a one-parameter subgroup if and only if
there exists usual derivative $r'(\theta_0).$ In general case, any indicated extremal is a metric straight line.
\end{proposition}
	
\begin{proof}
The first statement follows from (\ref{derst}). 	
		
In addition, by Theorem \ref{mt}, every admissible control $(u_1(t),u_2(t))=(u_1(\theta_0),u_2(\theta_0))$ 
from (\ref{u}) is constant if and only if there exists the usual derivative $r'(\theta_0),$ what is equivalent to
condition that the system (\ref{a4}) has unique solution, a one-parameter subgroup
$$x(t)=u_1(\theta_0)t,\quad y(t)=u_2(\theta_0)t,\quad z(t)\equiv 0,\quad v(t)\equiv 0.$$
				
Notice that there exists at most countable number of values $\theta_0,$ for which the second statement is false. 
For any such $\theta_0,$ $x(t),$ $y(t),$ $t\in\mathbb{R},$ are as in (\ref{xt}), (\ref{yot}) with
$\theta(\tau)\equiv \theta_0$ and arbitrary measurable integrands $u_1(\tau),$ $u_2(\tau)$ of the type,
indicated in Theorem \ref{mt}, and the functions $z(t)$ and $v(t)$ are defined by formulas (\ref{aa4}) and (\ref{aa5}) respectively.

It follows from (\ref{a4}) that the length of any arc for the curve $(x(t),y(t),z(t),v(t))$ in $(G,d)$ is equal to the 
length of corresponding arc for its projection $(x(t),y(t))$  on the Minkowski plane. One can easily see that projections 
of indicated curves are metric straight lines on the Minkowski plane. Therefore the curves itself are metric straight lines.  
\end{proof}
	
\begin{remark}
The metric straight lines are obtained only in the case of Proposition \ref{norm1}, in particular, Proposition \ref{ageod1}.	
\end{remark}
	
{\bf 2.2.} Let us consider the case $\varphi_4=0$, $\varphi_3\neq 0$.
	
\begin{proposition}
\label{q2}
Let $(x,y,z,v)(t)$, $t\in\mathbb{R}$, be an extremal with conditions $x(0)=y(0)=z(0)=v(0)=0$ on the Engel group such that
$\varphi_4=0$, $\varphi_3\neq 0$. Then the functions $\theta(t),$ $h(t)=(h_1(t),h_2(t))$, $x(t),$ $y(t)$ are periodic with common period $L=2S_0/|\varphi_3|,$ where  $S_0$ is the area of the figure $U^{\ast}.$ The 
projection $(x,y)(t)$ of the extremal onto the Minkowski plane $z=v=0$ with the norm $F$ has a form
\begin{equation}
\label{xy}
x(t)= \frac{h_2(t)-\varphi_2}{\varphi_3},\quad  y(t)=- \frac{h_1(t)-\varphi_1}{\varphi_3},
\end{equation} 
 and it is a parametrized by the arc length periodic curve on an isoperimetrix.  In addition, $h_1=h_1(\theta(t)),$ $h_2=h_2(\theta(t))$ are given by formulas (\ref{h}), $\theta=\theta(t)$ is the inverse function to the function  
 $t(\theta)= \int_{\theta_0}^{\theta}(r^2(\xi)/\varphi_3)d\xi,$
 $$z(t)=\frac{t-\varphi_1x(t)-\varphi_2y(t)}{2\varphi_3}$$ 
 and $z(t)$ is equal to oriented area on
 the Euclidean plane with the Cartesian coordinates $x$, $y$, traced by rectilinear segment connecting the origin with
 point $(x(\tau),y(\tau))$, $\tau\in [0,t]$. In addition, $v=v(t)$ is defined by formula (\ref{aa5}).
 \end{proposition}
	
\begin{proof}
The statements about the function  $\theta(t)$ follow from (\ref{derst}).	
It follows from (\ref{h34}) and (\ref{area}) that analogously to the second Kepler law the radius-vector-function 
$h(\tau)=(h_1(\tau),h_2(\tau))\in U^{\ast},$ $t_1\leq \tau\leq t_2,$ traces in the plane $h_1,h_2$ 
(or, if it is desired, $u_1,u_2$ or $x,y$) with the standard Euclidean oriented area $(\varphi_3/2)(t_2-t_1).$ 
Consequently, $h(t),$ $t\in\mathbb{R},$ is a periodic function with period $L=2S_0/|\varphi_3|,$ where $S_0$ is the area
of the figure $U^{\ast}.$ Moreover, (\ref{h34}), (\ref{a7}) and (\ref{a4}) imply
formulas (\ref{xy}), i.e. the projection $(x,y)(t)$ of the curve $(x,y,z,v)(t)$ lies on the boundary $I(\varphi_1,\varphi_2,\varphi_3)$ of
the figure obtained by rotation of the figure $U^{\ast}/|\varphi_3|$ by angle $\frac{\pi}{2}$ around the center 
(origin of coordinates) with consequent shift by vector $\left(-\frac{\varphi_2}{\varphi_3},\frac{\varphi_1}{\varphi_3}\right)$. 
Thus, analogously to the case of the Heisenberg group with left-invariant sub-Finsler metric, considered in \cite{Ber2}, $I(\varphi_1,\varphi_2,\varphi_3)$ is an {\it isoperimetrix of the Minkowski plane with the norm $F$} \cite{Lei}.
		
Analogously to \cite{Ber2}, (\ref{xy}) implies that $(x(t),y(t))$ is a periodic curve on 
$I(\varphi_1,\varphi_2,\varphi_3)$ with period $L$ indicated above. It follows from (\ref{eq}) и (\ref{xy}) that
\begin{equation}
\label{z} 
z(t)=(t-\varphi_1x(t)-\varphi_2y(t))/(2\varphi_3)=(\varphi_3t-\varphi_1h_2(t)+\varphi_2h_1(t))/(2\varphi_3^2),
\end{equation}
\begin{equation}
\label{zl}
z(L)=L/(2\varphi_3)=S_0/(|\varphi_3|\varphi_3).
\end{equation}
The statement of Proposition \ref{q2} on the function $z(t)$ follows from (\ref{a4}). 
Since $(x(t),y(t))$ lies on isoperimetrix passing clockwise (counterclockwise) if $\varphi_3<0$ ($\varphi_3>0$), then $z(t)$ 
is a monotone function. In particular, $z(L)$ is oriented area of the figure enveloped by isoperimetrix  $I(\varphi_1,\varphi_2,\varphi_3)$ or, what is the same, area of the figure $U^{\ast}/|\varphi_3|$ taken with the sign equal to
the sign of $z(L)$. 

The last statement was proved earlier.
\end{proof}
	
{\bf 2.3.} Assume that $\varphi_4\neq 0$.

\begin{lemma}
\label{tconst}

If the function  $\theta(t)$ is constant on some non-degenerate interval $J\subset \mathbb{R},$ then on $J$
\begin{equation}
\label{con1}
x(t)\equiv -\frac{\varphi_3}{\varphi_4},\quad y(t)=y_0+\frac{2\varphi_4(t-t_0)}{2\varphi_2\varphi_4-\varphi_3^2},\quad
z(t)=z_0+\frac{\varphi_3(t-t_0)}{\varphi_3^2-2\varphi_2\varphi_4},
\end{equation}
\begin{equation}
\label{con2}
v(t)=v_0+\frac{\varphi_3^2(t-t_0)}{6\varphi_4(2\varphi_2\varphi_4-\varphi_3^2)},
\end{equation}
where  $y_0=y(t_0)$, $z_0=z(t_0)$, $v_0=v(t_0),$ $t_0$ is a point of the interval $J$ closest to zero; 
$v_0$ is calculated by $x_0=-\frac{\varphi_3}{\varphi_4},$ $y_0,$ $z_0$ and (\ref{eq}) for $t=t_0$.
	
In particular, for $\varphi_3=0,$
\begin{equation}
\label{con3}
x(t)\equiv 0,\quad y(t)=y_0+\frac{t-t_0}{\varphi_2},\quad z(t)\equiv z_0,\quad v(t)\equiv v_0=(t_0-\varphi_2y_0)/(3\varphi_4).
\end{equation}
\end{lemma}
	
\begin{proof}
If the function  $\theta(t)$ is equal to $\theta_0$ on some interval $J$, then $\dot{\theta}(t)\equiv 0$ and $x(t)\equiv -\varphi_3/\varphi_4$,
$h_2(t)=\varphi_2-\varphi_3^2/(2\varphi_4)$, $t\in J$, due to (\ref{derst}) и (\ref{h12}). It follows from (\ref{a4}) and (\ref{m}) that then
$u_1(t)\equiv 0$, $u_2(t)=1/h_2(t)\equiv (2\varphi_2\varphi_4-\varphi_3^2)/2\varphi_4$, $t\in J$, and the function $y(t)$ on the interval $J$ is determined by the second equality (\ref{con1}). Due to the calculated value $u_2$ and (\ref{aa4}), the function $z(t)$ on the interval $J$
is determined by the third equality (\ref{con1}). The equality (\ref{con2}) follows from (\ref{aa5}) and the first equality in (\ref{con1}).
The equalities  (\ref{con3}) are consequences from (\ref{con1}), (\ref{con2}) и (\ref{eq}).
\end{proof}	

\begin{lemma}
\label{tconst2}	
Suppose that there exist $\Theta_1,\Theta_2\in\mathbb{R}$ such that $\Theta_1< \Theta_2$ and the right-hand side in  (\ref{dt}) is positive for $\theta\in (\Theta_1,\Theta_2)$ and vanishes for $\theta=\Theta_1$ and for $\theta=\Theta_2$, $\theta(t),$ $t\in\mathbb{R},$ is a function admissible by the Maximum Principle. Then $\varphi_4u_1(\theta)>0$ ($\varphi_4u_1(\theta)<0$) for $\theta\in (\Theta_1,\Theta_2),$ sufficiently
close to $\Theta_1$ (respectively, $\Theta_2$), (\ref{u}). If $\theta(t_0)=\Theta_1$ ($\theta(t_0)=\Theta_2$) and $\theta(t)\neq \Theta_1$ 
($\theta(t)\neq \Theta_2$) for all $t<t_0$ or $t>t_0$, sufficiently close to $t_0,$ then  $\dot{\theta}(t)(t-t_0)> 0$ ($\dot{\theta}(t)(t-t_0) <0$) for these $t$.
\end{lemma}

\begin{proof}
We note only that the second statement of Lemma is a consequence of the first one and (\ref{derst}).
\end{proof}		

\begin{theorem}
\label{mt1}	
If $\varphi_4\neq 0$ then any extremal on the Engel group starting at the unit is defined by the equations (\ref{xt}),
(\ref{yot}) (with arbitrary measureable integrands of indicated view and continuously differentiable function
$\theta=\theta(t)$ satisfying (\ref{derst}), (\ref{dt})),
\begin{equation}
\label{zt}
z(t)=-\frac{1}{\varphi_4}\left(r(\theta(t))\cos\theta(t)-\varphi_1+\left(\varphi_3+\frac{1}{2}\varphi_4x(t)\right)y(t)\right),
\end{equation}
\begin{equation}
\label{vt}
v(t)=\frac{1}{3\varphi_4}\left(t-\varphi_1x(t)-\varphi_2y(t)-\left(2\varphi_3+\frac{1}{2}\varphi_4x(t)\right)z(t)\right).
\end{equation}   	
	
Let us set
$$\theta_0:=\theta(0),\quad \mathcal{E}_0=\max_{h\in U^{\ast}}\left(-\varphi_4h_2\right),\quad 
\mathcal{E}_{-1}=\min_{h\in U^{\ast}}\left(-\varphi_4h_2\right).$$
The following cases are possible.
	
1. Let  $\varphi_3\neq 0$ and $\mathcal{E}>\mathcal{E}_0$. Then the function $\theta(t)$, $t\in\mathbb{R}$,  is inverse to the function $t(\theta)$
defined by formula
\begin{equation}
\label{tth} t(\theta)=\int_{\theta_0}^{\theta}\frac{r^2(\xi)d\xi}{\varphi_3\sqrt{1+(2\varphi_4/\varphi_3^2)(r(\xi)\sin\xi-\varphi_2)}}, \quad r(\theta_0)\sin\theta_0=\varphi_2.
\end{equation}
	
2. Let $\varphi_3=0$ and $\mathcal{E}=\mathcal{E}_{-1}$. Then $\theta(t)\equiv\theta_0$ and the desired extremal is the metric straight line (\ref{anorm}).
	
3. Let $\mathcal{E}_{-1}< \mathcal{E}<\mathcal{E}_0$. Then we have for some numbers $t_1$, $t_2,$ $t_1\neq t_2,$ for any $t\in\mathbb{R}$ and $k\in\mathbb{Z}$  
\begin{equation}
\label{dts00}
\theta(t+2k(t_2-t_1))=\theta(t),\quad\dot{\theta}(t_i+t)=-\dot{\theta}(t_i-t),\,\,\theta(t_i+t)=\theta(t_i-t),\,\,i=1,2.
\end{equation}
	
3.1. If $\varphi_3\neq 0$ then $t_i=t(\theta_i)$, $i=1,2$, in equalities (\ref{dts00}) are calculated by (\ref{tth}), where 
$\theta_1\neq\theta_2$ are the nearest to $\theta_0$ values such that $\varphi_3(\theta_2-\theta_1)>0$ and the right-hand side in (\ref{dt}) vanishes.
	
3.2. If $\varphi_3=0$ then $\theta_2\neq\theta_1=\theta_0$ and $t_1=0$, $t_2=t(\theta_2)$ in (\ref{dts00}), where
\begin{equation}
\label{tth0}
t(\theta)=\pm\int\limits_{\theta_0}^{\theta}
\frac{r^2(\xi)d\xi}{\sqrt{2\varphi_4( r(\xi)\sin{\xi}-\varphi_2)}},\quad r(\theta_0)\sin\theta_0=\varphi_2,
\end{equation}
and on the right--hand side stands $+$ (respectively, $-$) if $\varphi_4(\theta_2-\theta_0)>0$  (respectively, $\varphi_4(\theta_2-\theta_0)<0$). Here $\theta_2\neq\theta_0$ is a number such that $h_2(\theta_0)=h_2(\theta_2)$ and $\varphi_4(h_2(\theta)-h_2(\theta_0))>0$ for any $\theta$ from interval 
$I=(\min(\theta_0,\theta_2),\max(\theta_0,\theta_2))$.
	
4. Let  $\varphi_3\neq 0$ and $\mathcal{E}=\mathcal{E}_0$. Then there exist the nearest to $\theta_0$ values $\theta_1$, $\theta_2$  such that 
$\theta_1<\theta_0<\theta_2$ and the right-hand side in (\ref{dt}) vanishes for $\theta=\theta_i,$ $i=1,2$. 
If improper integral (\ref{tth}) diverges for $\theta=\theta_1$ and $\theta=\theta_2$, then  $\theta(t)\in (\theta_1,\theta_2),$ $t\in \mathbb{R},$ 
is the inverse function to the function $t(\theta)$ defined by (\ref{tth}). 
If at least one of improper integrals (\ref{tth0}) is finite for $\theta=\theta_1$ and(or) for $\theta=\theta_2$, then the function $\theta(t)$
is not unique and can take constant values on some non-degenerate closed intervals of arbitrary length, on which (\ref{con1}), (\ref{con2}) are valid.
		
5. Let $\varphi_3=0$ and $\mathcal{E}=\mathcal{E}_{0}$. Then there exists the largest segment $[\theta_1,\theta_2]$, $\theta_1\leq\theta_2$,
such that $\theta_0\in [\theta_1,\theta_2]$ and $h_2(\theta)=\varphi_2$ for any $\theta\in [\theta_1,\theta_2]$.
If $\theta_0=\theta_2$ (respectively, $\theta_0=\theta_1$) then further we denote by $t(\theta)$ the integral (\ref{tth}) for $\theta\in [\theta_0,\theta_1+2\pi]$ (respectively, $\theta\in[\theta_1-2\pi,\theta_0]$) without $+$ and $-$. 
	
Then $\theta(t)\equiv\theta_0$ and the desired extremal is the metric straight line (\ref{anorm}) in the following cases:
	
5.1. $\theta_1=\theta_0=\theta_2$ and $t(\theta)=\infty$ for $\theta\nearrow\theta_0$ and for $\theta\searrow\theta_0$;
	
5.2. $\theta_1<\theta_0<\theta_2$;
	
5.3. $\theta_0=\theta_1<\theta_2$ and $t(\theta)=\infty$ for $\theta\nearrow\theta_0$;
	
5.4. $\theta_0=\theta_2>\theta_1$ and $t(\theta)=\infty$ for $\theta\searrow\theta_0$.
	
In all other cases, the function $\theta(t)$ is not unique and can take constant values on some closed intervals of arbitrary length, on which (\ref{con3}) are valid.
\end{theorem}

\begin{proof}
The first statement follows from Theorem \ref{mt} and (\ref{h12}), (\ref{h}), (\ref{eq}).

Let us prove statements of theorem about the function $\theta(t).$ 
To do this, knowing the sign of the derivative $\dot{\theta}(\theta),$ we find $\dot{\theta}(\theta)$ from (\ref{dt}), where $\dot{\theta}$ denotes the derivative with respect to $t.$ Integrating this function, we will find  $\theta(t).$ 

1. Indicated conditions mean that $\varphi_3\neq 0$ and the right-hand side in (\ref{dt}) is positive for all $\theta\in\mathbb{R}$.
Then due to (\ref{derst}), (\ref{dt}),
\begin{equation}
\label{dth}
\dot{\theta}(\theta)=(\varphi_3/r^2(\theta))\sqrt{1+ (2\varphi_4/\varphi_3^2)(r(\theta)\sin\theta-\varphi_2)},
\end{equation}
whence follows (\ref{tth}). Moreover, on the ground of (\ref{derst}), (\ref{dth}),
\begin{equation}
\label{xot} 
x(t)=(\varphi_3/\varphi_4)(\sqrt{1+ (2\varphi_4/\varphi_3^2)(r(\theta(t))\sin\theta(t)-\varphi_2)}-1),
\end{equation}
and $y(t),$ $t\in\mathbb{R},$ is defined by (\ref{yot}).

2. Indicated conditions mean that $\varphi_3=0$ and the right-hand side in (\ref{dt}) is non-positive for all $\theta\in\mathbb{R}$.
Then $h_2(\theta_0)=\varphi_2$ is maximal (respectively, minimal) value of the second component for $h\in U^{\ast}$ if $\varphi_4>0$ (respectively, $\varphi_4< 0$), $\dot{\theta}\equiv 0$, $\theta(t)\equiv\theta_0$, and due to (\ref{derst}) we get $x(t)\equiv 0$ and 
the metric straight lines (\ref{anorm}).

We shall use Lemma \ref{tconst2} without mentions to prove the remaining statements.

3. Indicated conditions mean that the right-hand side in (\ref{dt}) takes both positive and negative values.

3.1. At first, let us consider the case $\varphi_3\neq 0.$
It is clear that there exist $\theta_1$, $\theta_2$ as in the statement of p.\,3 in Theorem  \ref{mt1}. 
In consequence of reflexiveness in passing to the dual normed space for finite-dimensional case,
for all values $\theta$ sufficiently close to $\theta_1$ (respectively, $\theta_2$), the first formula in (\ref{u}) defines values $u_1(\theta)\neq 0$ of the same sign, but of opposite signs for $\theta_1$ and $\theta_2.$ 
Then $\theta_1<\theta_0<\theta_2<\theta_1+2\pi$ if $\varphi_3>0$, 
$\theta_2<\theta_0<\theta_1<\theta_2+2\pi$ if $\varphi_3<0$, 
both values $t_i=t(\theta_i),$ $i=1,2,$ are finite and it is defined a monotone continuously differentiable function
$\theta(t),$ $t_1\leq t\leq t_2,$ with zero one-sided derivatives on the ends. According to what has been said, the equalities (\ref{dts00}) are valid which uniquely determine the function $\theta(t),$ $t\in\mathbb{R}.$

3.2. It is clear that in the case $\varphi_3=0$  there exists $\theta_2$ for $\theta_1=\theta_0$ as in the statement of p.\,3.2 in Theorem  \ref{mt1}. In consequence of the Taylor formula applied to the corresponding one-sided derivatives of the first order, 
subradical function in the denominator of integrand in (\ref{tth0}) has order equal to $1/2$ relative to
$(\xi-\theta_0)^2$ and $(\xi-\theta_2)^2$  when $\xi\in I$ and $\xi\rightarrow \theta_0,$ $\xi\rightarrow\theta_2$ 
respectively. Therefore, the function  $t(\theta)$ calculated by (\ref{tth0}) is finite for all $\theta\in \overline{I}$, $t_2>0$, and for $\varphi_4(\theta_2-\theta_0)>0$ (respectively, $\varphi_4(\theta_2-\theta_0)<0$) on the segment $[0,t_2]$ the increasing (respectively, decreasing) function $\theta(t)$ is defined. The function is inverse to the function $t(\theta)$, where on the right-hand side of formula (\ref{tth0}) stands $+$ (respectively, $-$). By reflexiveness, for all sufficiently close to $\theta_0$ (respectively, $\theta_2$) values $\theta\in I$ the first formula in (\ref{u}) defines $u_1(\theta)\neq 0$ of the same sign, but of opposite signs for $\theta_0$ and $\theta_2$. 
Therefore the function $\theta(t)\in \overline{I},$ $t\in\mathbb{R},$ is even, periodic with period $2t_2,$ alternately
increasing and decreasing on respective segments of length $t_2$, and relations (\ref{dts00}) are valid, uniquely determining the function $\theta(t)$, $t\in\mathbb{R}$. According to what has been said, the function $x(t)$ for $t\in [0,t_2]$ is defined by formula
\begin{equation}
\label{xfromt}
x(t)=({\rm sgn}(\dot{\theta}(t))/\varphi_4)\sqrt{2\varphi_4(r(\theta(t))\sin\theta(t)-\varphi_2)},
\end{equation}
$x(t)$ is odd,  $x(t+2kt_2)=x(t),$ $k\in\mathbb{Z},$ and $x(t+t_2)=-x(t_2-t),$ $t\in\mathbb{R}$. The function $y(t)$, $t\in\mathbb{R}$, is defined by formula (\ref{yot}).

4. Indicated conditions mean that $\varphi_3\neq 0$, the right-hand side in (\ref{dt}) is non-negative, and there exists the only
$h_2^0\neq \varphi_2$ with some $h=(h_1,h_2^0)\in \partial U^{\ast}$ such that the the right-hand side in (\ref{dt}) vanishes. It is clear that 
$h_2^0$ is minimal (respectively, maximal) value of the second component for points from $U^{\ast}$ if $\varphi_4> 0$ (respectively, $\varphi_4 < 0$) and $h_2^0= \varphi_2 - \varphi_3^2/2\varphi_4$ in consequence of  (\ref{derst}) and (\ref{h12}). In addition, the vector $h\in \partial U^{\ast}$ with $h_2=h_2^0$ is not unique if $\partial U^{\ast}$ is not strictly convex at points 
$h$ with $h_2=h_2^0$, in other words, if $\partial U$ is not differentiable at the point $u^0=(0,1/h_2^0)$. 
In any case there exist nearest to $\theta_0$ values $\theta_1<\theta_0$ and $\theta_2>\theta_0$ such that $r(\theta_i)\sin\theta_i=h_2^0,$ $i=1,2.$ 

4.1. If $t(\theta_i)=\pm\infty,$ $i=1,2,$ then $\theta(t)\in (\theta_1,\theta_2)$, $t\in \mathbb{R}$ is the inverse function to the function
$t(\theta)$ defined by (\ref{tth}). For example, this is true if subradical function in the denominator of integrand in (\ref{tth}) 
has orders not less than one relative to
$(\xi-\theta_1)^2,$ $(\xi-\theta_2)^2$  under $\xi\searrow \theta_1,$ $\xi\nearrow\theta_2$ 
respectively (which is satisfied if there exist the usual second derivatives $r''(\theta_i),$ $i=1,2$). Indicated
conditions may fail even under the existence of usual derivatives $r'(\theta_i),$ $i=1,2.$

4.2. Let $t_i:=t(\theta_i)$ be finite for $i=i_1$ and infinite for $i=i_2\neq i_1.$ Then  $\dot{\theta}(t_{i_1})=0,$ 
$t_{i_2}={\rm sgn}(\phi_3(i_2-i_1))\infty$ and in interval $I$ between $t_{i_1}$ and $t_{i_2}$ is defined
the function $\theta(t)$.  For all $t\in\mathbb{R}-I,$ the function could be defined as $\theta(t)=\theta(2t_{i_1}-t)$. 
It is possible also that $\theta(t)\equiv \theta_{i_1}$ if and only if
$t$ belongs to the closure of some nonempty open interval $I_1\subset \mathbb{R}-I,$ where 
$\overline{I_1}\cap \overline{I}=t_{i_1}.$ If $I_1$ is finite then the graph of the function $\theta(\tau)$
on $\mathbb{R}-(\overline{I}\cup \overline{I_1}):=I_2$ of the last solution is obtained from the graph of the first solution
by shift of interval $\mathbb{R}-\overline{I}$ to $I_2.$ 

4.3. Let both $t_i,$ $i=1,2$ are finite. Then $\dot{\theta}(t_1)=\dot{\theta}(t_2)=0$ and on segment $I$ with ends  $t_1\neq t_2$ is uniquely
defined the function $\theta=\theta(t)$. Under condition 4.3.1: $\theta_2 =\theta_1 + 2\pi,$  the graph
of a continuously differentiable function $\theta(t),$ $t\in \mathbb{R},$ can admit parts obtained from the graph of the function on $I$
or its reflection with respect to straight line $t=t_2$ by a combination of parallel vertical
shifts  by values, equal to $2k\pi$ for some $k\in \mathbb{Z},$ and parallel horizontal shifts, with adjacent closed intervals of arbitrary
lengths of constancy of the function $\theta=\theta(t)$. Under condition 4.3.2: $\theta_2<\theta_1 + 2\pi,$ 
are admissible all continuous functions $\theta(t),$ $t\in \mathbb{R},$ whose graphs on some segments of variable $t$ 
with length $|t_2-t_1|$ are horizontal shifts of its graph on $I$ or its reflection
with respect to $t=t_2$ with values $\theta_1$ and $\theta_2$ at ends of these segments, with adjacent closed intervals of arbitrary
lengths for constancy of the function $\theta=\theta(t)$.

5. Indicated conditions mean that $\varphi_3=0$ and the right-hand side in (\ref{dt}) is non-negative for all $\theta\in\mathbb{R}$. Then
$h_2(\theta_0)=\varphi_2$ is minimal (respectively, maximal) value of the second component for $h\in U^{\ast}$ if $\varphi_4> 0$ 
(respectively, $\varphi_4 < 0$).
In addition, the vector $h\in \partial U^{\ast}$ with $h_2=\varphi_2$ is not unique if $\partial U^{\ast}$ is not strictly convex at points 
$h$ with $h_2=\varphi_2$, in other words, if $\partial U$ is not differentiable at the point $u^0=(0,1/\varphi_2)$. 
In general, there exists the largest segment  $[\theta_1,\theta_2]$, $\theta_1\leq\theta_2$, such that $\theta_0\in [\theta_1,\theta_2]$ and $h_2(\theta)=\varphi_2$ for any $\theta\in [\theta_1,\theta_2]$.

Let us consider again the improper integral (\ref{tth0}). It is clear that $\theta(t)\equiv\theta_0$ and we obtain only one of two
metric straight lines (\ref{anorm}) in each case 5.1--5.4 of Theorem \ref{mt1}. In all other cases, there also may be such solutions.

We shall indicate {\it all other possible extremals} in remaining cases:

5.5.  $\theta_1=\theta_2=\theta_0$, $\tau(\theta)$ is finite for $\theta\searrow\theta_0$
and $\tau(\theta)=\infty$ for $\theta\nearrow\theta_0$;

5.6. $\theta_1=\theta_2=\theta_0$, $\tau(\theta)$ is finite for $\theta\nearrow\theta_0$
and $\tau(\theta)=\infty$ for $\theta\searrow\theta_0$;

5.7. $\theta_1<\theta_0=\theta_2$ and 
$\tau(\theta)$ is finite for $\theta\searrow\theta_0$;

5.8. $\theta_0=\theta_1<\theta_2$ and $\tau(\theta)$ is finite for $\theta\nearrow\theta_0$;

5.9. $\theta_1=\theta_2=\theta_0$, $\tau(\theta)$ is finite for $\theta\searrow\theta_0$ and for $\theta\nearrow\theta_0$.

In all these cases, it may be that $\theta(t)\equiv \theta_0$ on some finite or infinite closed interval $J$ including $0$.
Assume that $J\neq\mathbb{R}$ and $\theta(t)\neq \theta_0$ for $t\notin J$  close enough to $J.$ 

a) Let us consider the cases 5.7 under condition $t(\theta_1+2\pi)=\infty$ and 5.5 (the cases 5.8 with condition $\tau(\theta_2-2\pi)=\infty$ and 5.6)).
If $\sup(J)=t_2<+\infty$ then  on the interval $J_2$ between $t_2$ and $+\infty$ is defined the function $\theta(t)=\Theta(t-t_2)$, where 
$\Theta(t),$ $t\geq 0,$ is inverse to the function  $t(\theta)$ from (\ref{tth0}) with $+$ (respectively, $-$) on the right-hand-side. If $\inf(J)=t_1> -\infty$ then on the interval $J_1$ between $t_1$ and $-\infty$ is defined the function $\theta(t)=\Theta(t_1-t)$. 

With the same signs, on the ground of (\ref{derst}), (\ref{dth}), the functions $x(t)$ and $y(t)$ are defined by (\ref{xfromt}) and (\ref{yot}) respectively.

b) Let $t_3=t(\theta_1+2\pi)$ be finite in the case 5.7. 
Then $t_3>0$ and the function $\Theta(t),$ $t\in I=[0,t_3],$  is determined which is inverse to the function $t(\theta)$ from (\ref{tth0}) with $+$ on the right-hand-side. 
All continuous functions $\theta(t),$ $t\in \mathbb{R},$ with $\theta(0)=\theta_0$ are admissible whose graphs on some segments of the variable $t$ with length $t_3$ are horizontal shifts of the graph of the function $\Theta(t)$ on $I$ or its reflection relative to the vertical line $t=t_3$, with adjacent closed intervals of arbitrary
lengths on which $\theta$ takes constant values $\theta_0$ or $\theta_1+2\pi.$ In all considered cases, the functions $x(t)$ and $y(t)$ are defined by (\ref{xfromt}) and (\ref{yot}) respectively. The case 5.8 with finite $t_2:=t(\theta_2-2\pi)$ is considered in similar way.

5.9. Assume now that $\theta_1=\theta_2=\theta_0$, $t(\theta)$ is finite for $\theta\searrow\theta_0,$ and  $t_3:=t(\theta_1+2\pi)$. Then
$t_3>0$ and all continuously differentiable functions $\theta(t),$ $t\in \mathbb{R},$ are admissible whose  graphs contain parts obtained from the graph of the function $\Theta(t)$ on $I$ from 5.7b) or its reflection relative to the straight line $t=t_3$ by combinations of
vertical parallel shifts by values equal to $2k\pi$ for some  $k\in \mathbb{Z}$ and horizontal parallel shifts,  with adjacent closed intervals with arbitrary lengths of constancy of $\theta$. 
\end{proof}

\section{About cases with square control regions}

The norm $F_{\alpha}$ on $D(e)$ is defined by formula
\begin{equation}
\label{norma}
F_{\alpha}(u_1,u_2)=\max\{|u_1\cos\alpha+u_2\sin\alpha|,|-u_1\sin\alpha+u_2\cos\alpha|\},\quad \alpha\in [0,\pi/2).
\end{equation} 

\begin{remark}
In paper \cite{AS5}, such metrics were considered for $0\leq\alpha\leq\pi/4$.
\end{remark}

The unit ball $U_{\alpha}=\{(u_1,u_2):\,F_{\alpha}(u_1,u_2)\leq 1\}$ is obtained from the 
unit ball $U:=U_{0}$ of the norm $F(u_1,u_2)=\max\{|u_1|,|u_2|\}$ by rotation by the angle $\alpha,$ and
$\partial U_{\alpha}$ is a described square around the unit circle $S^1=\{(u_1,u_2):\, u_1^2+ u_2^2=1\}$ with four tangency points.
The polar curve $U_{\alpha}^{\ast}$ is the convex hull of these points and $\partial U_{\alpha}^{\ast}$ is the isoperimetrix to
the Minkowski plane $(D(e),F_{\alpha}).$ 

The following proposition holds.

\begin{proposition}
\label{polars}
The polar equation of the square $\partial U_{\alpha}^{\ast}$ has a form
$$r(\theta)=\frac{\sqrt{2}}{2\cos\left(\theta-\alpha-\frac{\pi}{4}\right)},\quad \alpha\leq \theta\leq \alpha+\frac{\pi}{2},\quad
r\left(\theta+\frac{\pi}{2}\right)=r(\theta),\quad \theta\in\mathbb{R}.$$
In addition,
$$r'(\theta)=\frac{\sqrt{2}\sin\left(\theta-\alpha-\frac{\pi}{4}\right)}{2\cos^2\left(\theta-\alpha-\frac{\pi}{4}\right)},\quad\alpha < \theta < \alpha+\frac{\pi}{2},$$
$$ -1\leq r'(\alpha)\leq 1,\quad r'\left(\theta+\frac{\pi}{2}\right)=r'(\theta),\quad \theta\in\mathbb{R}.$$
\end{proposition}

{\bf 1.}  It follows from (\ref{norma}) that $F_{\alpha}(0,1)=\cos\alpha$ if $0\leq\alpha\leq\pi/4,$ and $F_{\alpha}(0,1)=\sin\alpha$ if 
$\pi/4<\alpha <\pi/2;$ in abnormal case, the equations (\ref{anorm}) has the forms
$$x(t)=z(t)=v(t)\equiv 0,\quad y(t)=
\left\{\begin{array}{rl}
\pm\frac{t}{\cos\alpha},\,\,\text{if }\,\,0\leq\alpha\leq\pi/4,\\
\pm\frac{t}{\sin\alpha},\,\,\text{if }\,\,\pi/4<\alpha <\pi/2.
\end{array}\right.$$

{\bf 2.1.} The pair $(\varphi_1,\varphi_2)=r(\theta_0)(\cos\theta_0,\sin\theta_0)$ is in $\partial U_{\alpha}^{\ast}$. 

At first assume that the point $(\varphi_1,\varphi_2)$ is not a vertex of the square $\partial U_{\alpha}^{\ast}$. Then (\ref{u}) gives unique solution $(u_1(t),u_2(t))=(u_1(\theta_0),u_2(\theta_0))$ that is a vertex of the square $\partial U_{\alpha}$ such that
$\angle((u_1(\theta_0),u_2(\theta_0)),(\varphi_1,\varphi_2))< \pi/4,$ and the system (\ref{a4}) has the only solution, namely the one-parameter subgroup 
$$x(t)=u_1(\theta_0)t,\quad y(t)=u_2(\theta_0)t,\quad z(t)\equiv 0,\quad v(t)\equiv 0.$$ 

Now let the point  $(\varphi_1,\varphi_2)$ be one of the vertices of the square $\partial U_{\alpha}^{\ast}$.
In this case, there exists a segment $\Delta$ (a side of the square $\partial U_{\alpha}$) of solutions $(u_1,u_2)=(u_1(\theta_0),u_2(\theta_0))$ to equations (\ref{u}) such that $\angle((u_1,u_2),(\varphi_1,\varphi_2))\leq \pi/4.$
Every measu\-rable function $(u_1(t),u_2(t))\in\Delta$ defines a curve (\ref{xt}), (\ref{yot}), (\ref{aa4}), (\ref{aa5}). 

In any case, we get only metric  straight lines.

\begin{example}
If $\alpha=\varphi_2=\varphi_3=\varphi_4=0,$ $\varphi_1=1,$ then $\theta(t)\equiv \theta_0=0,$ $t\in\mathbb{R},$ arbitrary vector--function of the kind
\begin{equation}
\label{exam}
u(t)=(1,u_2(t)),\quad -1\leq u_2(t)\leq 1, \quad t\in\mathbb{R},
\end{equation} 
with measurable real function $u_2=u_2(t)$, and the covector function $\psi(t)=(1,0,0,0)$ satisfy the Pontryagin Maximum Principle for $t\in\mathbb{R}$,
moreover the corresponding (extremal) trajectory $g=g(t)$, $t\in\mathbb{R}$, with origin $g(0)=e$ is the metric stright line. 
Thus, in general case, when searching for extremals and even geodesics of left-invariant sub-Finsler metric on Lie group, it is not possible to exclude the control from the Hamiltonian system for the Pontryagin Maximum Principle.
This statement is true for any non-strictly convex control region $U\subset D(e),$ in other words, for the polar figure $U^{\ast}$ with  non-differentiable boundary $\partial U^{\ast}.$ 
\end{example}

{\bf 2.2.} Arguing as in the proof of Proposition \ref{q2}, we get that $h(t)$, $u(t)$, $x(t)$, $y(t)$, $t\in\mathbb{R}$, are periodic functions with period $L=4/|\varphi_3|$ and equalities (\ref{xy}), (\ref{z}), (\ref{zl}) and (\ref{aa5}) hold, i.e., the projection $(x,y)(t)$ of the curve $(x,y,z,v)(t)$
lies on a square obtained from $\partial U_{\alpha}^{\ast}/|\varphi_3|$ by shifting its center to 
the point $\left(-\frac{\varphi_2}{\varphi_3},\frac{\varphi_1}{\varphi_3}\right)$. The control is piecewise constant on the complement to a countable set of isolated points.

{\bf 2.3.} Cases 1, 2, 3.1 and 3.2 of Theorem \ref{mt1} are possible. The case 4 of Theorem \ref{mt1} is possible only for the option 4.3 considered in the proof of Theorem \ref{mt1}, moreover, the option 4.3.2 is possible only for $\alpha=\pi/4$. The case 5 of Theorem \ref{mt1} is possible for the options  5.2, 5.7b), 5.8b) (for $\alpha=\pi/4$), 5.9 (for $\alpha\neq\pi/4$), considered in the proof of Theorem \ref{mt1}. 
The description of all possible options for the function $\theta(t),$ $t\in\mathbb{R},$ in these cases are given in the proof of Theorem \ref{mt1}, the functions $x(t)$ and $y(t)$ are found by (\ref{xt}), (\ref{yot}) with usage of Proposition \ref{polars}, and the functions
$z(t),$ $v(t),$ $t\in\mathbb{R},$ are found by (\ref{zt}), (\ref{vt}).  

\section{Extremals of left-invariant sub-Finsler quasimetric on the Engel group}

The proofs and results of our paper are valid also for the case of a left-invariant sub-Finsler
quasimetric on the Engel group. Quasimetrics have all properties of metric, except possibly symmetry property $d(p,q)=d(q,p).$
For this, we need to make only the following changes in the text:

1) As $U$, we take an arbitrary convex (two-dimensional) figure, containing inside $0$, perhaps $U\neq -U$.

2) Instead of references to the reflexiveness in passing to the dual
normed vector space for the finite-dimensional case we must refer to a theorem on the bipolar figure ($U^{\ast\ast}=U$, see Theorem 14.5 in \cite{Rock}).

\end{document}